\documentclass[11pt, amsfonts, dvipdfmx]{amsart}

\usepackage{amsmath,amssymb,amsthm,color,enumerate}
\usepackage{mathtools}
\usepackage[all]{xy}
\usepackage[usenames]{xcolor}
\definecolor{refkey}{rgb}{0.6, 0.7, 0.4}
\definecolor{labelkey}{rgb}{0, 0.7, 0.5}
\usepackage{todonotes}
\SelectTips{eu}{12}

\newtheorem{thm}{Theorem}[section]
\newtheorem{prop}[thm]{Proposition}
\newtheorem{lem}[thm]{Lemma}
\newtheorem{cor}[thm]{Corollary}

\theoremstyle{definition}

\newtheorem{prob}[thm]{Problem}
\theoremstyle{remark}

\numberwithin{equation}{section}

\newcommand{\pt}{\mathrm{pt}}

\newcommand{\ZZ}{\mathbb{Z}}
\newcommand{\st}{\mathrm{star}}
\newcommand{\lk}{\mathrm{link}}

\makeatletter
\@namedef{subjclassname@2020}{%
  \textup{2020} Mathematics Subject Classification}
\makeatother

\title{Independence complexes of $(n \times 4)$ and $(n \times 5)$-grid graphs}

\author[T. Matsushita]{Takahiro Matsushita}
\address{Department of Mathematical Sciences, University of the Ryukyus, Nishihara-cho, Okinawa 903-0213, Japan}
\email{mtst@sci.u-ryukyu.ac.jp}

\author[S. Wakatsuki]{Shun Wakatsuki}
\address{Graduate School of Mathematics, Nagoya University, Furocho, Chikusaku,
Nagoya, 464-8602, Japan}
\email{shun.wakatsuki@math.nagoya-u.ac.jp}

\keywords{independence complexes; square grid graphs}
\subjclass[2020]{55P10; 05C69}

\begin{document}

\baselineskip.525cm

\maketitle

\begin{abstract}
We determine the homotopy types of the independence complexes of $(n \times 4)$ and $(n \times 5)$-square grid graphs. In fact, they are homotopy equivalent to wedges of spheres.
\end{abstract}

\section{Introduction}

Let $G = (V,E)$ be a finite simple graph. A subset $\sigma$ of $V$ is {\it independent} if no two vertices in $\sigma$ are not adjacent in $G$. The family of independent sets in $G$ forms a simplicial complex, and we call it the {\it independence complex $I(G)$ of $G$}.

The independence complex is one of the most principal constructions of simplicial complexes. In fact, a simplicial complex $K$ is isomorphic to some independence complex if and only if $K$ is a clique complex. A \emph{clique complex} is a simplicial complex such that the dimension of every minimal non-face is at most one, and has appeared in several branches of mathematics. The order complexes of posets, appearing in combinatorics \cite{Kozlov book}, and the Vietoris--Rips complexes, appearing in topological data analysis and geometric group theory \cite{PRSZ}, are typical examples of clique complexes. The matching complex of a graph $G$ is the independence complex of the line graph of $G$, and its precise homotopy type has been recently studied by several authors (see \cite{AGS}, \cite{BH}, \cite{BGJM}, \cite{Matsushita}, and \cite{Wachs}). Since the barycentric subdivision of every simplicial complex is a clique complex, for every simplicial complex $K$ there exists an independence complex homeomorphic to $K$.

The homotopy types of independence complexes have been extensively studied (see \cite{Adamaszek1}, \cite{Barmak}, \cite{BMLN}, \cite{EH}, \cite{Engstrom}, \cite{Iriye}).
In general, it is quite difficult to determine the homotopy type of the independence complex $I(G)$ even though the graph $G$ is simply described.
In this paper, we study the homotopy types of the independence complexes of grid graphs $\Gamma_{n,k}$. Here, for a pair $n$ and $k$ of positive integers, define the \emph{$(n \times k)$-grid graph $\Gamma_{n,k}$} by
\[ V(\Gamma_{n,k}) = \big\{ (x,y) \in \ZZ^2 \; | \; 1 \le x \le n, \; 1 \le y \le k \big\},\]
\[ E(\Gamma_{n,k}) = \big\{ \{ (x,y), (z,w) \} \; | \; (x,y), (z,w) \in V(\Gamma_{n,k}), \; |x-z| + |y-w| = 1 \big\}.\]

The goal of this paper is to determine the precise homotopy types of $I(\Gamma_{n,4})$ and $I(\Gamma_{n,5})$.

We now review known results related to $I(\Gamma_{n,k})$. When $k=1$, the graph $\Gamma_{n,1}$ is the path graph $P_n$ consisting of $n$ vertices, and the homotopy types of $I(\Gamma_{n,1})$ is determined by Kozlov \cite{Kozlov}. The homotopy types of $I(\Gamma_{n,2})$ and $I(\Gamma_{n,3})$ are determined by Adamaszek \cite{Adamaszek2}. In these cases, $I(\Gamma_{n,k})$ is contractible or homotopy equivalent to a sphere. Recently, the authors \cite{MW} determined the homotopy types of $I(\Gamma_{n,6})$, which are homotopy equivalent to wedges of spheres. Otherwise, the homotopy types of $I(\Gamma_{n,k})$ are little known.
%However, for $k \ge 4$ with $k \ne 6$, the homotopy types of $I(\Gamma_{n,k})$ are not determined.
%For $k \ge 4$, the homotopy types of $I(\Gamma_{n,k})$ are not completely determined.
However, Okura \cite{Okura} determined the Euler characteristic of $I(\Gamma_{n,4})$, and it is an unbounded function with respect to $n$. His computation implies that for large $n$, $I(\Gamma_{n,4})$ is neither contractible nor homotopy equivalent to a sphere. Bousquet-M\'elou, Linusson, and Nevo \cite{BMLN} studied the independence complexes of several types of square grid graphs.

Our grid graph $\Gamma_{n,k}$ is the cartesian product $P_n \times P_k$ of the path graphs $P_n$ and $P_k$. The homotopy types of independence complexes of $C_n \times P_k$ and $C_n \times C_k$ have also been studied by several authors, motivated by statistical physics. A celebrated result by Jonsson \cite{Jonsson1} asserts that if $n$ and $k$ are coprime, then the reduced Euler characteristic of $I(C_n \times C_k)$ is $-1$. His computation showed that the homotopy types of $I(P_n \times P_k)$ and $I(C_n \times C_k)$ seem to be very different since the absolute value of the Euler characteristic of $I(P_n \times P_4) = I(\Gamma_{n,4})$ tends to infinity (see \cite{Okura} or Theorem \ref{main thm n x 4}). See \cite{Adamaszek2}, \cite{Iriye}, \cite{Jonsson2}, \cite{Jonsson3}, and \cite{Thapper} for more information about these complexes.

Now we state the main results in this paper.

\begin{thm} \label{main thm n x 4}
For $n \ge 1$, there is a following homotopy equivalence:
\[ I(\Gamma_{n,4}) \simeq \begin{cases}
\bigvee_{2k} S^{n-1} & (n = 6k+1) \\
\bigvee_{2k+2} S^{n-1} & (n = 6k + 4) \\
\bigvee_{2k+1} S^{n-1} & (\textrm{$n = 6k + l$ with $l = 0,2,3,5$}).
\end{cases}\]
\end{thm}

\begin{thm} \label{main thm}
  There are following homotopy equivalences
\[ I(\Gamma_{1,5}) \simeq S^1, \; I(\Gamma_{4,5}) \simeq S^4, \; I(\Gamma_{5,5}) \simeq S^5, \; I(\Gamma_{9,5}) \simeq S^{10}.\]
  Assume that $n$ is a positive integer which is not equal to $1,4,5,9$. Let $t$, $l$, and $\varepsilon$ be non-negative integers satisfying $n = 20t + 4l + \varepsilon$, $l \in \{ 0,1,2,3,4\}$, and $\varepsilon \in \{ 0,1,2,3\}$.
  Define $\nu = \lfloor \varepsilon / 2\rfloor \in \{0, 1\}$ (i.e.,\ the greatest integer less than or equal to $\varepsilon / 2$)
  and $n' = 25t + 5l + \varepsilon - 1 + \nu$.
  Then the following hold:%\todo{$\bigvee$ 付近に括弧つけた方が良い？}
  \begin{enumerate}[$(1)$]
    \item Assume that $n = 0, 4, 5, 9, 10, 14, 15, 19$ modulo $20$. Then
      \[
      I(\Gamma_{n,5}) \simeq \bigvee_{3-2\nu} S^{n'}
      \vee \Big( \bigvee_{n'-t-\nu< i<n'} \Big( \bigvee_4 S^i \Big) \Big)
      \vee \bigvee_2 S^{n'-t-\nu}.
      \]
    \item Assume that $n = 2, 3, 6, 7, 8, 11, 12, 13, 16, 17$ modulo $20$. Then
      \[
      I(\Gamma_{n,5}) \simeq \bigvee_{3-2\nu} S^{n'}
      \vee \Big( \bigvee_{n'-t\le i<n'} \Big(\bigvee_4 S^i\Big) \Big).
      \]
    \item Assume that $n = 1, 18$ modulo $20$. Then
      \[
      I(\Gamma_{n,5}) \simeq \bigvee_{3-2\nu} S^{n'}
      \vee \Big( \bigvee_{n'-t-2\nu< i<n'} \Big( \bigvee_4 S^i\Big) \Big).
      \]
  \end{enumerate}
% \begin{enumerate}[$(1)$]
% \item Assume that $n = 0, 4,5,9$ modulo $20$. Then
% \[ I(\Gamma_{n,5}) \simeq \bigvee_3 S^{25t + 5l + \varepsilon - 1} \vee \Big( \bigvee_4 S^{25t + 5l + \varepsilon -2} \vee \cdots \vee \bigvee_4 S^{24t + 5l + \varepsilon}\Big) \vee \bigvee_2 S^{24t + 5l + \varepsilon - 1}.\]

% \item Assume that $n = 1,8,12,13,16,17$ modulo $20$. Then
% \[ I(\Gamma_{n,5}) \simeq \bigvee_3 S^{25t + 5l + \varepsilon - 1} \vee \Big( \bigvee_4 S^{25t + 5l + \varepsilon -2} \vee \cdots \vee \bigvee_4 S^{24t + 5l + \varepsilon}\Big).\]

% \item Assume that $n = 10, 14,15,19$ modulo $20$. Then
% \[ I(\Gamma_{n,5}) \simeq S^{25t + 5l + \varepsilon - 1} \vee \Big( \bigvee_4 S^{25t + 5l + \varepsilon - 2} \vee \cdots \vee \bigvee_4 S^{24t + 5l + \varepsilon}\Big) \vee \bigvee_2 S^{24t + 5l + \varepsilon - 1}.\]

% \item Assume that $n = 2,3, 6,7, 11$ modulo $20$. Then
% \[ I(\Gamma_{n,5}) \simeq S^{25t + 5l + \varepsilon -1} \vee \Big( \bigvee_4 S^{25t + 5l + \varepsilon - 2} \vee \cdots \vee \bigvee_4 S^{24t + 5l + \varepsilon} \Big).\]

% \item Assume that $n = 18$ modulo $20$. Then
% \[ I(\Gamma_{n,5}) \simeq S^{25t + 5l + \varepsilon - 1} \vee \Big( \bigvee_4 S^{25t + 5l + \varepsilon -2} \vee \cdots \vee \bigvee_4 S^{24t + 5l + \varepsilon -1}\Big).\]
% \end{enumerate}
\end{thm}

\begin{cor} \label{main cor}
The Euler characteristic of $I(\Gamma_{n,5})$ is even and its absolute value is at most $4$.
%The Euler characteristic of $I(\Gamma_{n,5})$ is a bounded function whose Euler characteristic is even and whose absolute value is at most $4$.
\end{cor}

Discrete Morse theory is a commonly used method to determine the homotopy types of independence complexes (see \cite{BMLN} and \cite{Thapper} for example). However, it is very difficult to use discrete Morse theory to determine the homotopy type of a simplicial complex which is not a wedge of spheres of the same dimension. For other methods to determine the homotopy types of independence complexes, see \cite{Adamaszek1}, \cite{Barmak}, and \cite{EH}. Our method is an elementary method using links of simplicial complexes, reviewed in Section 2.

Finally, we pose a problem related to the Euler characteristic of $I(\Gamma_{n,k})$, inspired by Theorem \ref{main thm n x 4} and Corollary \ref{main cor}.

\begin{prob}
Given a positive integer $k$, determine whether the function assigning $\chi(I(\Gamma_{n,k}))$ to a positive integer $n$ is a bounded function.
\end{prob}

Recently, the authors \cite{MW} showed that the function assigning $\chi(I(\Gamma_{n,6}))$ to $n$ is bounded.

The rest of this paper is organized as follows. In Section 2, we review definitions and facts related to independence complexes. Sections 3 and 4 are devoted to the proofs of Theorems \ref{main thm n x 4} and \ref{main thm}, respectively.

\section{Preliminaries}

In this section, we review several definitions and facts related to independence complexes which we need in the subsequent sections. For a comprehensive introduction to simplicial complexes, we refer to \cite{Kozlov book}. In the latter of this section, we provide an alternative proof of determining the homotopy type of $I(\Gamma_{n,3})$.

A \emph{(finite simple) graph} is a pair $(V,E)$ such that $V$ is a finite set and $E$ is a subset of the family $\binom{V}{2}$ of $2$-element subsets of $V$. A subset $S$ of $V$ is \emph{independent} or \emph{stable} if there are no $v,w \in S$ such that $\{ v,w \} \in E$. The \emph{independence complex of a graph $G$} is the simplicial complex whose vertex set is $V$ and whose simplices are independent sets in $G$, and is denoted by $I(G)$.

A \emph{subgraph of $G = (V,E)$} is a graph $G' = (V', E')$ such that $V' \subset V$ and $E' \subset E$. The subgraph $G'$ of $G$ is said to be \emph{induced} if $E' = V' \cap E$. For a subset $S$ of $V$, we write $G - S$ to mean the induced subgraph of $G$ whose vertex set is $V - S$. Note that if $G'$ is an induced subgraph then there is a natural inclusion from $I(G')$ to $I(G)$.

If $G$ is a disjoint union of two subgraphs $G_1$ and $G_2$, then $I(G)$ coincides with the join $I(G_1) * I(G_2)$. In particular, if $G$ has an isolated vertex, then $I(G)$ is contractible. Since $I(K_2) = S^0$, we have $I(K_2 \sqcup G) \simeq \Sigma I(G)$. Here, $\Sigma$ denotes the suspension.

Let $K$ be an (abstract) simplicial complex and $S$ a subset of $V(K)$. Let $K - S$ denote the simplicial complex consisting of the simplices of $K$ disjoint from $S$. For a face $\sigma$ of $K$, define the \emph{star} $\st(\sigma)$ and the \emph{link} $\lk(\sigma)$ by
\[ \st(\sigma) = \{ \tau \in K \; | \; \sigma \cup \tau \in K \},\; \lk(\sigma) = \st(\sigma) - \sigma.\]
For $v \in V$, we write $\st(v)$ and $\lk(v)$ instead of $\st(\{ v\})$ and $\lk (\{ v\})$, respectively.

\begin{lem}
  Let $K$ be a simplicial complex and $v$ a vertex of $K$.
  Then $K$ is homeomorphic to the mapping cone of the inclusion $\lk(v) \hookrightarrow K - v$,
  and hence there is a cofiber sequence
  \[\lk(v)\to K - v \to K.\]
  In particular, if the inclusion $\lk (v) \hookrightarrow K - v$ is null-homotopic,
  then $K \simeq (K - v) \vee \Sigma  \lk (v)$.
\end{lem}

Let $G$ be a graph and let $v \in V(G)$. Let $N(v)$ be the set of vertices of $G$ adjacent to $v$, and set $N[v] = N(v) \cup \{ v\}$. It is straightforward to see $I(G) - v = I(G - v)$ and $\lk(v) = I(G - N[v])$.
Applying the above lemma to $K=I(G)$, we have the following:

\begin{thm}[see \cite{Adamaszek1}] \label{thm cofiber}
Let $G$ be a graph and $v$ be a vertex of $G$. Then there is a cofiber sequence
\[ I(G - N[v]) \to I(G - v) \to I(G).\]
In particular, if the inclusion $I(G - N[v]) \to I(G - v)$ is null-homotopic, then $I(G) \simeq I(G - v) \vee \Sigma I(G - N[v])$.
\end{thm}

This theorem implies the following simple argument, which will be frequently used.

\begin{cor}[fold lemma \cite{Engstrom}] \label{cor fold lemma}
Let $G$ be a graph, and $v$, $w$ vertices of $G$. Assume that $v \ne w$ and $N(v) \subset N(w)$. Then the inclusion $I(G - w) \hookrightarrow I(G)$ is a homotopy equivalence.
\end{cor}
\begin{proof}
We write the proof for the reader's convenience. By Theorem \ref{thm cofiber}, it suffices to see that $v$ is an isolated vertex of $G - N[w]$. If $v \in N[w]$, then $v \ne w$ means that $v \in N(w)$. This means $w \in N(v) \subset N(w)$. This is a contradiction. Hence $v \in V(G - N[w])$. Since $N(v) \subset N(w)$, we have that $v$ is an isolated vertex in $G - N[w]$. This completes the proof.
\end{proof}

The fold lemma immediately determines the homotopy types of $I(\Gamma_{n,1})$ and $I(\Gamma_{n,2})$:

\begin{thm}[Kozlov \cite{Kozlov}] \label{thm nx1}
There is a following homotopy equivalence:
\[ I(\Gamma_{n,1}) = I(P_n) \simeq
\begin{cases}
	S^{k-1} & (n = 3k) \\
	{\rm pt} & (n = 3k + 1) \\
	S^k & (n = 3k + 2).
\end{cases} \]
\end{thm}

\begin{thm}[Adamaszek \cite{Adamaszek2}] \label{thm nx2}
There is a following homotopy equivalence:
\[ I(\Gamma_{n,2}) \simeq \begin{cases}
	S^{k-1} & (n = 2k) \\
	S^k & (n = 2k +1).
\end{cases}\]
\end{thm}

We need the computation of the homotopy type of $I(C_n)$, where $C_n$ denotes the $n$-cycle graph.

\begin{thm}[Kozlov \cite{Kozlov}] \label{thm cycle}
For $n \ge 3$, there is a following homotopy equivalence:
\[ I(C_n) \simeq \begin{cases}
S^{k-1} \vee S^{k-1} & (n = 3k) \\
S^{k-1} & (n = 3k+1) \\
S^k & (n = 3k + 2).
\end{cases}\]
\end{thm}

\subsection{$I(\Gamma_{n,3})$}

In this subsection, we provide an alternative proof of the following theorem because our proofs in the subsequent sections are further developments of this proof:

\begin{thm}[Adamaszek \cite{Adamaszek2}] \label{thm n x 3}
For $l = 0,1,2,3$, there is a following homotopy equivalence:
\[ I(\Gamma_{4k+l,3}) \simeq S^{3k+l - 1}.\]
\end{thm}

Throughout this subsection, we write $\Gamma_n$ to mean $\Gamma_{n,3}$, and set $X_n = \Gamma_n - (n,2)$ and $Y_n = \Gamma_n - \{ (n,1), (n,3)\}$. This notation is valid only in this subsection, and in the subsequent sections we use the symbols $\Gamma_n$, $X_n$, and $Y_n$ to indicate other graphs.

We first determine the homotopy type of $I(X_n)$. By the fold lemma (Corollary \ref{cor fold lemma}), we have the following homotopy equivalences:
\begin{eqnarray} \label{eqn 2.1}
I(X_n) \simeq I(K_2 \sqcup K_2 \sqcup Y_{n-2}) = \Sigma^2 I(Y_{n-2}), I(Y_n) \simeq I(K_2 \sqcup X_{n-2}) = \Sigma I(X_{n-2}).
\end{eqnarray}
Hence the homotopy type of $I(X_n)$ is inductively determined by $I(X_1)$, $I(X_2)$, $I(Y_1)$, and $I(Y_2)$. Using the fold lemma, we have
\begin{eqnarray} \label{eqn 2.2}
I(X_1) \simeq \pt, I(X_2) \simeq S^1, I(Y_1) = \pt, I(Y_2) \simeq S^0.
\end{eqnarray}

\begin{lem} \label{lem X_n in section 2}
There is a following homotopy equivalence:
\[ I(X_n) \simeq
\begin{cases}
	{\rm pt} & (\textrm{$n$ is odd}) \\
	S^{3k-1} & (n = 4k) \\
	S^{3k+1} & (n = 4k+2).
\end{cases} \]
\end{lem}
\begin{proof}
	Using equation \eqref{eqn 2.1}, we have that $I(X_n) \simeq \Sigma^3 I(X_{n-4})$ for $n > 4$, $I(X_3) \simeq \Sigma^2 I(Y_1) = \pt$, and $I(X_4) \simeq \Sigma^2 I(Y_2) \simeq S^1$ by equations \eqref{eqn 2.1} and \eqref{eqn 2.2}. This completes the proof.
\end{proof}

\begin{proof}[Proof of Theorem \ref{thm n x 3}]
	Applying Theorem \ref{thm cofiber} to the vertex $v = (n,2)$ in $\Gamma_n = \Gamma_{n,3}$, we have a cofiber sequence
	\[ I(X_{n-1}) \to I(X_n) \to I(\Gamma_n).\]
	By Lemma \ref{lem X_n in section 2}, we have that $I(X_{n-1}) \to I(X_n)$ is null-homotopic since either $I(X_{n-1})$ or $I(X_n)$ is contractible. Hence we have
	\[
	I(\Gamma_n) \simeq I(X_n) \vee \Sigma I(X_{n-1}).
	\]
	Lemma \ref{lem X_n in section 2} completes the proof.
\end{proof}

\begin{figure}[t]
\begin{picture}(220,100)(0,-20)
\multiput(40,0)(0,20){4}{\circle*{3}}
\multiput(60,0)(0,20){4}{\circle*{3}}
\multiput(80,0)(0,60){2}{\circle*{3}}

\multiput(40,0)(20,0){2}{\line(0,1){60}}
\multiput(80,0)(0,60){2}{\line(-1,0){50}}
\multiput(60,20)(0,20){2}{\line(-1,0){30}}

\put(45,-20){$X_n$}
\put(10,27){$\cdots$}

\multiput(160,0)(0,20){4}{\circle*{3}}
\multiput(180,0)(0,20){4}{\circle*{3}}
\multiput(200,20)(0,20){2}{\circle*{3}}

\multiput(160,0)(20,0){2}{\line(0,1){60}}
\put(200,20){\line(0,1){20}}
\multiput(180,0)(0,60){2}{\line(-1,0){30}}
\multiput(200,20)(0,20){2}{\line(-1,0){50}}

\put(165,-20){$Y_n$}
\put(130,27){$\cdots$}
\end{picture}
\caption{}  \label{figure 4XY}
\end{figure}

\section{$I(\Gamma_{n,4})$}

The goal of this section is to prove Theorem \ref{main thm n x 4}, which determines the precise homotopy type of $I(\Gamma_{n,4})$. Throughout this section, we write $\Gamma_n$ to mean $\Gamma_{n,4}$.

\begin{figure}[b]
\begin{picture}(260,100)(0,-40)
\multiput(20,0)(20,0){5}{\circle*{3}}
\multiput(20,20)(20,0){4}{\circle*{3}}
\multiput(20,40)(20,0){4}{\circle*{3}}
\multiput(20,60)(20,0){5}{\circle*{3}}

\multiput(20,0)(20,0){4}{\line(0,1){60}}
\multiput(100,0)(0,60){2}{\line(-1,0){90}}
\multiput(80,20)(0,20){2}{\line(-1,0){70}}

\put(60,-20){$X_n$}
\put(50,-40){$I(X_n)$}
\put(-10,27){$\cdots$}

\multiput(180,0)(0,20){4}{\circle*{3}}
\multiput(200,0)(0,20){4}{\circle*{3}}
\multiput(220,20)(0,20){2}{\circle*{3}}

\multiput(180,0)(20,0){2}{\line(0,1){60}}
\put(220,20){\line(0,1){20}}
\multiput(200,0)(0,60){2}{\line(-1,0){30}}
\multiput(220,20)(0,20){2}{\line(-1,0){50}}

\multiput(220,0)(0,60){2}{\circle{3}}
\multiput(240,20)(0,20){2}{\circle{3}}

\multiput(240,0)(0,60){2}{\circle*{3}}
\multiput(260,0)(0,60){2}{\circle*{3}}
\multiput(240,0)(0,60){2}{\line(1,0){20}}

\put(155,-20){$Y_{n-2} \sqcup K_2 \sqcup K_2$}
\put(145,-40){$I(Y_{n-2} \sqcup K_2 \sqcup K_2) = \Sigma^2 I(Y_{n-2})$}
\put(150,27){$\cdots$}

\put(120,-20){$\xhookleftarrow{\mathmakebox[15pt]{}}$}
\put(120,-40){$\xhookleftarrow{\mathmakebox[15pt]{\simeq}}$}

\put(120,27){$\xhookleftarrow{\mathmakebox[15pt]{}}$}
\end{picture}
\caption{}  \label{4 X to Y}
\end{figure}

First, define the graphs $X_n$ and $Y_n$ by $X_n = \Gamma_n - \{ (n,2), (n,3)\}$ and $Y_n = \Gamma_n - \{ (n,1), (n,4)\}$ (see Figure \ref{figure 4XY}). The fold lemma implies
\begin{eqnarray}
I(X_n) \simeq \Sigma^2 I(Y_{n-2}),\; I(Y_n) \simeq \Sigma I(X_{n-1}).
\end{eqnarray}
See Figures \ref{4 X to Y} and \ref{4 Y to X}. We first determine the homotopy type of $I(X_n)$.

\begin{figure}[t]
\begin{picture}(260,100)(0,-40)
\multiput(20,0)(20,0){4}{\circle*{3}}
\multiput(20,20)(20,0){5}{\circle*{3}}
\multiput(20,40)(20,0){5}{\circle*{3}}
\multiput(20,60)(20,0){4}{\circle*{3}}

\multiput(20,0)(20,0){4}{\line(0,1){60}}
\multiput(100,20)(0,20){2}{\line(-1,0){90}}
\multiput(80,00)(0,60){2}{\line(-1,0){70}}
\put(100,20){\line(0,1){20}}

\put(60,-20){$Y_n$}
\put(50,-40){$I(Y_n)$}
\put(-10,27){$\cdots$}

\multiput(180,0)(0,20){4}{\circle*{3}}
\multiput(200,0)(0,20){4}{\circle*{3}}
\multiput(220,0)(0,20){4}{\circle*{3}}
\multiput(240,0)(0,60){2}{\circle*{3}}
\multiput(260,20)(0,20){2}{\circle*{3}}
\multiput(240,20)(0,20){2}{\circle{3}}

\multiput(180,0)(20,0){3}{\line(0,1){60}}
\multiput(240,0)(0,60){2}{\line(-1,0){70}}
\multiput(220,20)(0,20){2}{\line(-1,0){50}}
\put(260,20){\line(0,1){20}}

\put(160,-20){$X_{n-1} \sqcup K_2$}
\put(150,-40){$I(X_{n-1} \sqcup K_2) = \Sigma I(X_{n-1})$}
\put(150,27){$\cdots$}

\put(120,27){$\xhookleftarrow{\mathmakebox[15pt]{}}$}

\put(120,-20){$\xhookleftarrow{\mathmakebox[15pt]{}}$}
\put(120,-40){$\xhookleftarrow{\mathmakebox[15pt]{\simeq}}$}
\end{picture}
\caption{}  \label{4 Y to X}
\end{figure}

\begin{lem} \label{lem X_n for n x 4}
There is a following homotopy equivalence:
\[ I(X_n) \simeq
\begin{cases}
S^{n-1} & (n = 3k, 3k+2) \\
* & (n = 3k+1).
\end{cases}\]
In particular, $I(X_n)$ is homotopy equivalent to a wedge of $S^{n-1}$.
\end{lem}
\begin{proof}
The fold lemma implies $I(X_1) \simeq *$, $I(X_2) \simeq S^1$, and $I(X_3) \simeq \Sigma^2 I(Y_1) \simeq S^2$. For $n \ge 4$, we have $I(X_n) \simeq \Sigma^2 I(Y_{n-2}) \simeq \Sigma^3 I(X_{n-3})$. This completes the proof.
\end{proof}

\begin{figure}[b]
\begin{picture}(260,100)(0,-40)
\multiput(40,0)(0,20){4}{\circle*{3}}
\multiput(60,0)(0,20){4}{\circle*{3}}
\multiput(80,0)(0,60){2}{\circle*{3}}
\multiput(100,0)(0,20){3}{\circle*{3}}
\put(100,60){\circle{3}}

\put(100,0){\line(-1,0){70}}
\put(100,0){\line(0,1){40}}
\multiput(60,20)(0,20){2}{\line(-1,0){30}}
\put(80,60){\line(-1,0){50}}

\multiput(40,0)(20,0){2}{\line(0,1){60}}

\put(60,-20){$G_n - v$}
\put(50,-40){$I(G_n - v)$}
\put(10,27){$\cdots$}

\multiput(180,0)(0,20){4}{\circle*{3}}
\multiput(200,0)(0,20){4}{\circle*{3}}
\multiput(220,0)(0,60){2}{\circle*{3}}
\multiput(240,20)(0,20){2}{\circle*{3}}
\multiput(240,0)(0,60){2}{\circle{3}}

\multiput(180,0)(20,0){2}{\line(0,1){60}}
\multiput(220,0)(0,60){2}{\line(-1,0){50}}
\multiput(200,20)(0,20){2}{\line(-1,0){30}}

\put(240,20){\line(0,1){20}}

\put(175,-20){$X_{n-1} \sqcup K_2$}
\put(165,-40){$I(X_{n-1} \sqcup K_2)$}
\put(150,27){$\cdots$}

\put(120,27){$\xhookleftarrow{\mathmakebox[15pt]{}}$}

\put(120,-20){$\xhookleftarrow{\mathmakebox[15pt]{}}$}
\put(120,-40){$\xhookleftarrow{\mathmakebox[15pt]{\simeq}}$}
\end{picture}
\caption{}  \label{figure 4 G - v}
\end{figure}

\begin{lem} \label{lem 2 of n x 4}
The following hold:
\begin{itemize}
	\item[(1)] $I(\Gamma_n)$ is homotopy equivalent to a wedge of $S^{n-1}$.
	\item[(2)] For $n \ge 3$, there is a homotopy equivalence $I(\Gamma_n) \simeq \Sigma I(X_{n-1}) \vee \Sigma^2 I(\Gamma_{n-2})$.
\end{itemize}
\end{lem}
\begin{proof}
We simultaneously prove both (1) and (2) by induction on $n$. The fold lemma implies $I(\Gamma_1) \simeq \pt$ and $I(\Gamma_2) \simeq S^1$. Hence (1) holds for $n = 1,2$.

Suppose $n \ge 3$. Set $G_n = \Gamma_n - \{ (n-1,2), (n-1,3)\}$ for $n \ge 1$. The fold lemma implies $I(\Gamma_n) \simeq I(G_n)$. Set $v = (n,4)$. Then we have $I(G_n - v) \simeq \Sigma I(X_{n-1})$ (see Figure \ref{figure 4 G - v}) and $I(G_n - N[v]) \simeq \Sigma I(\Gamma_{n-2})$ for $n \ge 3$ (see Figure \ref{figure 4 G - N[v]}). Hence Theorem \ref{thm cofiber} shows that there is a following cofiber sequence:
\begin{eqnarray} \label{eq key cofib}
\Sigma I(\Gamma_{n-2}) \to \Sigma I(X_{n-1}) \to I(\Gamma_n).
\end{eqnarray}
Lemma \ref{lem X_n for n x 4} implies that $\Sigma I(X_{n-1})$ is a wedge of $(n-1)$-spheres, and $\Sigma I(\Gamma_{n-2})$ is a wedge of $(n-2)$-spheres by the inductive hypothesis. Hence the map $\Sigma I(\Gamma_{n-2}) \to I(X_{n-1})$ is null-homotopic, and we have $I(\Gamma_n) \simeq \Sigma I(X_{n-1}) \vee \Sigma^2 I(\Gamma_{n-2})$. This completes the proof of (2) in the case of $n$. Since $\Sigma I(X_{n-1})$ is a wedge of $(n-1)$-spheres and $\Sigma I(\Gamma_{n-2})$ is a wedge of $(n-2)$-spheres, the cofiber sequence \eqref{eq key cofib} implies that $I(\Gamma_n)$ is a wedge of $(n-1)$-spheres. This completes the proof of (1) in the case of $n$.
\end{proof}

\begin{figure}[t]
\begin{picture}(260,100)(0,-40)
\multiput(40,0)(0,20){4}{\circle*{3}}
\multiput(60,0)(0,20){4}{\circle*{3}}
\put(80,0){\circle*{3}}
\put(80,60){\circle{3}}
\multiput(100,0)(0,20){2}{\circle*{3}}
\multiput(100,40)(0,20){2}{\circle{3}}

\multiput(40,0)(20,0){2}{\line(0,1){60}}
\multiput(60,20)(0,20){3}{\line(-1,0){30}}
\put(100,0){\line(-1,0){70}}
\put(100,0){\line(0,1){20}}

\put(40,-20){$G_n - N[v]$}
\put(30,-40){$I(G_n - N[v])$}
\put(10,27){$\cdots$}

\multiput(180,0)(0,20){4}{\circle*{3}}
\multiput(200,0)(0,20){4}{\circle*{3}}
\multiput(220,0)(0,60){2}{\circle{3}}
\multiput(240,40)(0,20){2}{\circle{3}}
\multiput(240,0)(0,20){2}{\circle*{3}}

\multiput(180,0)(20,0){2}{\line(0,1){60}}
\multiput(200,0)(0,20){4}{\line(-1,0){30}}

\put(240,0){\line(0,1){20}}

\put(175,-20){$\Gamma_{n-2} \sqcup K_2$}
\put(165,-40){$I(\Gamma_{n-2} \sqcup K_2)$}
\put(150,27){$\cdots$}

\put(120,27){$\xhookleftarrow{\mathmakebox[15pt]{}}$}

\put(120,-20){$\xhookleftarrow{\mathmakebox[15pt]{}}$}
\put(120,-40){$\xhookleftarrow{\mathmakebox[15pt]{\simeq}}$}
\end{picture}
\caption{} \label{figure 4 G - N[v]}
\end{figure}

\begin{lem} \label{lem 3 of n x 4}
For $n \ge 7$, there is a homotopy equivalence:
\[ I(\Gamma_n) \simeq S^{n-1} \vee S^{n-1} \vee \Sigma^6 I(\Gamma_{n-6}).\]
\end{lem}
\begin{proof}
By (2) of Lemma \ref{lem 2 of n x 4}, we have
\begin{eqnarray*}
I(\Gamma_n) & \simeq & \Sigma I(X_{n-1}) \vee \Sigma^2 I(\Gamma_{n-2}) \\
& \simeq & \Sigma I(X_{n-1}) \vee \Sigma^3 I(X_{n-3}) \vee \Sigma^4 I(\Gamma_{n-4}) \\
& \simeq & \Sigma I(X_{n-1}) \vee \Sigma^3 I(X_{n-3}) \vee \Sigma^5 I(X_{n-5}) \vee \Sigma^6 I(\Gamma_{n-6}) \\
& \simeq & S^{n-1} \vee S^{n-1} \vee \Sigma^6 I(\Gamma_{n-6}).
\end{eqnarray*}
The last homotopy equivalence follows from Lemma \ref{lem X_n for n x 4}.
\end{proof}

\begin{proof}[Proof of Theorem \ref{main thm n x 4}]
Lemma \ref{lem 3 of n x 4} implies that the homotopy type of $I(\Gamma_n)$ is inductively determined by $I(\Gamma_1), \cdots, I(\Gamma_6)$. Now we determine the homotopy types of these complexes.

The fold lemma (Corollary \ref{cor fold lemma}) implies that
\[ I(\Gamma_1) \simeq \pt, \;  I(\Gamma_2) \simeq S^1.\]
For $n \ge 3$, Lemma \ref{lem X_n for n x 4} and (2) of Lemma \ref{lem 2 of n x 4} implies
\[ I(\Gamma_3) \simeq \Sigma I(X_2) \vee \Sigma^2 I(\Gamma_1) \simeq S^2, \; I(\Gamma_4) \simeq \Sigma I(X_3) \vee \Sigma^2 I(\Gamma_2) \simeq S^3 \vee S^3,\]
\[ I(\Gamma_5) \simeq \Sigma I(X_4) \vee \Sigma^2 I(\Gamma_3) \simeq S^4, \; I(\Gamma_6) \simeq \Sigma I(X_5) \vee \Sigma^2 I(\Gamma_4) \simeq S^5 \vee S^5 \vee S^5.\]
Lemma \ref{lem 3 of n x 4} completes the proof.
\end{proof}

\section{$I(\Gamma_{n,5})$}

The goal of this section is to prove Theorem \ref{main thm}, which determines the precise homotopy type of $I(\Gamma_{n,5})$. Throughout this section, we write $\Gamma_n$ to mean $\Gamma_{n,5}$. Since the proof is fairly intricate, we first explain its outline.

\subsection{Outline of the proof}

\begin{figure}[t]
\begin{picture}(260,120)(0,-20)
\multiput(40,0)(0,20){5}{\circle*{3}}
\multiput(60,0)(0,20){5}{\circle*{3}}

\multiput(80,20)(20,0){2}{\circle*{3}}
\multiput(80,60)(20,0){2}{\circle*{3}}
\multiput(100,0)(20,0){4}{\circle*{3}}
\multiput(100,80)(20,0){4}{\circle*{3}}

\multiput(160,20)(0,20){3}{\circle*{3}}
\multiput(180,20)(20,0){2}{\circle*{3}}
\multiput(180,60)(20,0){2}{\circle*{3}}
\multiput(200,0)(20,0){4}{\circle*{3}}
\multiput(200,80)(20,0){4}{\circle*{3}}
\multiput(260,20)(0,20){3}{\circle*{3}}

\multiput(40,0)(20,0){2}{\line(0,1){80}}
\multiput(60,0)(0,20){5}{\line(-1,0){30}}
\multiput(60,20)(0,40){2}{\line(1,0){40}}
\multiput(100,0)(0,60){2}{\line(0,1){20}}
\multiput(100,0)(0,80){2}{\line(1,0){60}}
\multiput(160,0)(100,0){2}{\line(0,1){80}}
\multiput(200,0)(0,80){2}{\line(1,0){60}}
\multiput(160,20)(0,40){2}{\line(1,0){40}}
\multiput(200,0)(0,60){2}{\line(0,1){20}}

\put(10,37){$\cdots$}

\put(53,-13){\footnotesize $(n,1)$}
\put(140,-20){$A_{n,2}$}
\end{picture}
\caption{} \label{figure A}
\end{figure}

We introduce the graph $A_{n,k}$ recursively by $k$ as follows. First, set $A_{n,0} = \Gamma_n$. Suppose that $A_{n,k}$ has been defined as an induced subgraph of $\Gamma_{n + 5k}$. Then $A_{n,k+1}$ is defined as the induced subgraph of $\Gamma_{n + 5k + 5}$ whose vertex set is
\[ V(A_{n,k}) \cup (\{ n''+1, n''+2, n''+5\} \times \{ 1,3\}) \cup (\{ n''+2, n''+3, n''+4, n''+5 \} \times \{ 1,5\}) \cup \{ (n''+5, 3)\},\]
where $n''$ means $n + 5k$. Figure \ref{figure A} depicts $A_{n,2}$.
%Note that $A_{n, k+1}$ is constructed by attaching $P_{17}$ to $A_{n, k}$.
In our proof, we need to determine the homotopy type of not only $I(\Gamma_n)$ but also $I(A_{n,k})$.

For $n \ge 3$, let $v$ be the vertex $(n-2,3)$ of $A_{n,k}$. Theorem \ref{thm cofiber} implies that there is a cofiber sequence
\begin{eqnarray} \label{eq principal cofiber}
I(A_{n,k} - N[v]) \to I(A_{n,k} - v) \to I(A_{n,k}).
\end{eqnarray}
We use this cofiber sequence to determine the homotopy type of $I(A_{n,k})$ inductively. We first show the following proposition, which will be proved in Subsection 4.4.

\begin{prop} \label{prop A - N[v]}
Assume $n \ge 6$. Then there is a following homotopy equivalence:
\[ I(A_{n,k} - N[v]) \simeq I(A_{n-5, k+1}).\]
\end{prop}

By Proposition \ref{prop A - N[v]} and cofiber sequence \eqref{eq principal cofiber}, we have a following cofiber sequence
\begin{eqnarray} \label{main cofiber}
I(A_{n-5, k+1}) \to I(A_{n,k} - v) \to I(A_{n,k}),
\end{eqnarray}
The following proposition, which will be proved in Subsection 4.5, is a key to the whole proof:

\begin{prop} \label{prop null-homotopic}
For $n \ge 6$, the inclusion $I(A_{n-5,k+1}) \simeq I(A_{n,k} - N[v]) \to I(A_{n,k} - v)$ is null-homotopic. Hence
\[ I(A_{n,k}) \simeq I(A_{n,k} - v) \vee \Sigma I(A_{n-5,k+1}).\]
\end{prop}

Hence to determine the homotopy types of $A_{n,k}$, it suffices to determine the homotopy types of $I(A_{n,k} - v)$ for every pair $n$ and $k$, and $I(A_{n,k})$ for $n = 1, \cdots, 5$. The following two propositions, which will be proved in Subsections 4.2 and 4.3, assert that we can determine the homotopy types of these complexes.

\begin{prop} \label{prop A - v}
The following hold:
\begin{enumerate}[$(1)$]
\item Assume $k = 0,1$. Then
\[ I(A_{n,k} - v) \simeq \begin{cases}
{\rm pt} & \textrm{($n$ is odd)} \\
S^{5k + 5l - 1} & \textrm{($n = 4l$)} \\
S^{5k + 5l + 2} & \textrm{($n = 4l+2$).}
\end{cases}\]

\item Assume $k \ge 2$. Then
\[ I(A_{n,k} - v) \simeq \begin{cases}
{\rm pt} & \textrm{($n$ is odd)} \\
S^{5k + 5l - 1} \vee S^{5k+5l-1} & \textrm{($n = 4l$)} \\
S^{5k + 5l + 2} \vee S^{5k+5l+2} & \textrm{($n = 4l+2$).}
\end{cases}\]
\end{enumerate}
\end{prop}

\begin{prop} \label{prop 1-5}
The following hold:
\begin{enumerate}[$(1)$]
\item \[I(A_{1,k}) \simeq \begin{cases}
S^1 & (k = 0) \\
{\rm pt} & (k \ge 1)
\end{cases}\]

\item
\[ I(A_{2,k}) \simeq \begin{cases}
S^{5k+2} & (k = 0,1) \\
S^{5k+2} \vee S^{5k+2} & (k \ge 2)
\end{cases}\]

\item
\[ I(A_{3,k}) \simeq \begin{cases}
S^{5k+3} & (k=0) \\
S^{5k+3} \vee S^{5k+3} & (k \ge 1)
\end{cases}\]

\item
\[ I(A_{4,k}) \simeq \begin{cases}
S^{5k+4} & (k=0,1) \\
S^{5k+4} \vee S^{5k+4} & (k \ge 2)
\end{cases} \]

\item
\[ I(A_{5,k}) \simeq \begin{cases}
S^{5k+5} & (k=0) \\
S^{5k+5} \vee S^{5k+5} & (k \ge 1)
\end{cases} \]
\end{enumerate}
\end{prop}

The rest of this section is organized as follows. Subsections 4.2, 4.3, 4.4, and 4.5 are devoted to the proofs of Propositions \ref{prop A - v}, \ref{prop 1-5}, \ref{prop A - N[v]}, and \ref{prop null-homotopic}, respectively. In Subsection 4.6, we complete the proof of Theorem \ref{main thm}.

\subsection{Proposition \ref{prop A - v}}

The goal of this subsection is to show Proposition \ref{prop A - v}, which determines the homotopy type of $I(A_{n,k} - v)$. Before proceeding the proof, we need to determine the homotopy types of independence complexes of graphs $X_n$ (see Lemma \ref{lem X}) and $B_{n,k}$ (see Lemma \ref{lem B}).

\begin{figure}[t]
\begin{picture}(260,120)(0,-20)
\multiput(40,0)(0,20){5}{\circle*{3}}
\multiput(60,0)(0,20){5}{\circle*{3}}
\multiput(80,0)(0,20){5}{\circle*{3}}
\multiput(100,0)(0,40){3}{\circle*{3}}

\multiput(40,0)(20,0){3}{\line(0,1){80}}
\multiput(100,0)(0,40){3}{\line(-1,0){70}}
\multiput(80,20)(0,40){2}{\line(-1,0){50}}

\put(55,-20){$X_n$}
\put(10,37){$\cdots$}

\multiput(180,0)(0,20){5}{\circle*{3}}
\multiput(200,0)(0,20){5}{\circle*{3}}
\multiput(220,0)(0,20){5}{\circle*{3}}
\multiput(240,20)(0,40){2}{\circle*{3}}

\multiput(180,0)(20,0){3}{\line(0,1){80}}
\multiput(220,0)(0,40){3}{\line(-1,0){50}}
\multiput(240,20)(0,40){2}{\line(-1,0){70}}

\put(195,-20){$Y_n$}
\put(150,37){$\cdots$}

%\put(130,37){$\simeq$}

\end{picture}
\caption{} \label{XY}
\end{figure}

For $n \ge 1$, let $X_n$ be the graph $\Gamma_n - \{ (n,2), (n,4)\}$ and $Y_n$ the graph $\Gamma_n - \{ (n,1), (n,3), (n,5)\}$ (see Figure \ref{XY}). We first determine the homotopy type of $I(X_n)$. By the fold lemma (Corollary \ref{cor fold lemma}), it is clear that
\begin{eqnarray} \label{eqn 4.1}
I(X_n) \simeq \Sigma^3 I(Y_{n-2}), \; I(Y_n) \simeq \Sigma^2 I(X_{n-2})
\end{eqnarray}
for $n \ge 3$. Hence for $n \ge 5$, we have
\begin{eqnarray} \label{eqn X}
I(X_n) \simeq \Sigma^5 I(X_{n-4}).
\end{eqnarray}
Thus the homotopy type of $I(X_n)$ is determined by the homotopy types of $I(X_1), \cdots, I(X_4)$. The fold lemma implies
\[ I(X_1) \simeq \pt, I(X_2) \simeq S^2, \;  I(Y_1) \simeq \pt, I(Y_2) \simeq S^1.\]
Hence we have
\[ I(X_3) \simeq \Sigma^3 I(Y_1) \simeq \pt, \; I(X_4) \simeq \Sigma^3 I(Y_2) \simeq S^4.\]
Thus \eqref{eqn X} implies the following lemma:

\begin{lem} \label{lem X}
$I(X_n) \simeq \begin{cases}
\pt & (\textrm{$n$ is odd}) \\
S^{5l - 1} & (n = 4l) \\
S^{5l + 2} & (n = 4l + 2).
\end{cases}$
\end{lem}

Next we determine the homotopy types of the independence complexes of other families of graphs $B_{n,k}$ and $B'_{n,k}$. For $k \ge 1$, define the graph $B_{n,k}$ to be the induced subgraph of $A_{n,k}$ whose vertex set is $V(A_{n,k}) - V(\Gamma_n)$. For $k \ge 1$, define $B'_{n,k}$ to be the induced subgraph of $B_{n,k}$ whose vertex set is
\[ \big( V(B_{n,k}) - V(\Gamma_{n + 5}) \big) \cup \{ (n+5, 2), (n+5, 3), (n+5, 4)\}.\]
 Figures \ref{figure B} and \ref{figure B'} depict the graphs $B_{n,3}$ and $B'_{n,3}$, respectively. Note that the isomorphism types of graphs $B_{n,k}$ and $B'_{n,k}$ only depend on $k$.
 
 \begin{figure}[t]
\begin{picture}(330,120)(0,-20)
\multiput(20,20)(20,0){2}{\circle*{3}}
\multiput(20,60)(20,0){2}{\circle*{3}}
\multiput(40,0)(20,0){4}{\circle*{3}}
\multiput(40,80)(20,0){4}{\circle*{3}}
\multiput(100,20)(0,20){3}{\circle*{3}}
\multiput(120,20)(20,0){2}{\circle*{3}}
\multiput(120,60)(20,0){2}{\circle*{3}}
\multiput(140,0)(20,0){4}{\circle*{3}}
\multiput(140,80)(20,0){4}{\circle*{3}}
\multiput(200,20)(0,20){3}{\circle*{3}}
\multiput(220,20)(20,0){2}{\circle*{3}}
\multiput(220,60)(20,0){2}{\circle*{3}}
\multiput(240,0)(20,0){4}{\circle*{3}}
\multiput(240,80)(20,0){4}{\circle*{3}}
\multiput(300,20)(0,20){3}{\circle*{3}}

\multiput(20,20)(0,40){2}{\line(1,0){20}}
\multiput(40,0)(0,60){2}{\line(0,1){20}}

\multiput(40,0)(100,0){3}{\line(1,0){60}}
\multiput(40,80)(100,0){3}{\line(1,0){60}}
\multiput(100,0)(100,0){3}{\line(0,1){80}}

\multiput(100,20)(0,40){2}{\line(1,0){40}}
\multiput(140,0)(0,60){2}{\line(0,1){20}}

\multiput(200,20)(0,40){2}{\line(1,0){40}}
\multiput(240,0)(0,60){2}{\line(0,1){20}}

\put(160,-20){$B_{n,3}$}

\end{picture}
\caption{} \label{figure B}
\end{figure}

\begin{figure}[b]
\begin{picture}(440,120)(0,-20)
\multiput(100,20)(0,20){3}{\circle*{3}}
\multiput(120,20)(20,0){2}{\circle*{3}}
\multiput(120,60)(20,0){2}{\circle*{3}}
\multiput(140,0)(20,0){4}{\circle*{3}}
\multiput(140,80)(20,0){4}{\circle*{3}}
\multiput(200,20)(0,20){3}{\circle*{3}}
\multiput(220,20)(20,0){2}{\circle*{3}}
\multiput(220,60)(20,0){2}{\circle*{3}}
\multiput(240,0)(20,0){4}{\circle*{3}}
\multiput(240,80)(20,0){4}{\circle*{3}}
\multiput(300,20)(0,20){3}{\circle*{3}}

\multiput(140,0)(100,0){2}{\line(1,0){60}}
\multiput(140,80)(100,0){2}{\line(1,0){60}}
\multiput(200,0)(100,0){2}{\line(0,1){80}}
\put(100,20){\line(0,1){40}}

\multiput(100,20)(0,40){2}{\line(1,0){40}}
\multiput(140,0)(0,60){2}{\line(0,1){20}}

\multiput(200,20)(0,40){2}{\line(1,0){40}}
\multiput(240,0)(0,60){2}{\line(0,1){20}}

\put(205,-20){$B'_{n,3}$}

\end{picture}
\caption{} \label{figure B'}
\end{figure}

\begin{lem} \label{lem B}
The following hold:
\begin{enumerate}[$(1)$]
  % \item For $k \ge 2$, $I(B'_{n,k}) \simeq S^{5k-5} \vee S^{5k-5}$.
  % \item $I(B_{n,1}) \simeq S^4$.
  % \item For $k \ge 2$, $I(B_{n,k}) \simeq S^{5k-1} \vee S^{5k-1}$.
  \item \[
    I(B'_{n,k}) \simeq
    \begin{cases}
      S^0 & (k = 1) \\
      S^{5k-5}\vee S^{5k-5} & (k\ge 2)
    \end{cases}
    \]
  \item \[
    I(B_{n,k}) \simeq
    \begin{cases}
      S^4 & (k = 1) \\
      S^{5k-1}\vee S^{5k-1} & (k\ge 2)
    \end{cases}
    \]
\end{enumerate}
\end{lem}
\begin{proof}
  % Since $B_{n,1}$ is isomorphic to $P_{15}$, (2) is a special case of Theorem \ref{thm nx1}.
  For $k\ge 1$, the fold lemma implies $I(B_{n,k}) \simeq \Sigma^4 I(B'_{n,k})$.
  Hence (1) implies (2). Thus it suffices to show (1).

  Since $B'_{n,1}$ is isomorphic to $P_3$,
  we have $I(B'_{n,1})\simeq I(P_3)\simeq S^0$ by Theorem \ref{thm nx1}.
  Now we assume $k\ge 2$.
  Set $x = (n+5,4)$. Then the fold lemma shows that $I(B'_{n,k} - x) \simeq \Sigma^5 I(B'_{n,k-1} - x)$ (see Figure \ref{figure B' - x}), and the induction on $k$ implies $I(B'_{n,k} - x) \simeq S^{5k-5}$. On the other hand, $I(B'_{n,k} - N[x]) \simeq \Sigma^4 I(B'_{n,k-1} - x) \simeq S^{5k-6}$ (see Figure \ref{figure B' - N[x]}). Hence the inclusion $I(B'_{n,k} - N[x]) \to I(B'_{n,k} - x)$ is null-homotopic. Thus we have
  \begin{eqnarray*} % \label{eqn B'}
    I(B'_{n,k}) \simeq I(B'_{n,k} - x) \vee \Sigma I(B'_{n,k} - N[x]) \simeq S^{5k - 5} \vee S^{5k-5}.
  \end{eqnarray*}
  This completes the proof.
  % The proof is completed by the combination of \eqref{eqn B} and \eqref{eqn B'}.
\end{proof}

\begin{figure}[t]
\begin{picture}(390,120)(0,-40)
\multiput(20,20)(0,20){2}{\circle*{3}}
\put(20,60){\circle{3}}
\multiput(40,20)(20,0){2}{\circle*{3}}
\multiput(40,60)(20,0){2}{\circle*{3}}
\multiput(60,0)(20,0){4}{\circle*{3}}
\multiput(60,80)(20,0){4}{\circle*{3}}
\multiput(120,20)(0,20){3}{\circle*{3}}
\multiput(140,20)(20,0){2}{\circle*{3}}
\multiput(140,60)(20,0){2}{\circle*{3}}
\multiput(160,0)(0,80){2}{\circle*{3}}

\put(10,59){\footnotesize $x$}

\put(20,20){\line(0,1){20}\line(1,0){40}}
\put(40,60){\line(1,0){20}}
\multiput(60,0)(0,60){2}{\line(0,1){20}}
\multiput(60,0)(0,80){2}{\line(1,0){60}}
\put(120,0){\line(0,1){80}}
\multiput(120,20)(0,40){2}{\line(1,0){40}}
\multiput(160,0)(0,60){2}{\line(0,1){20}}
\multiput(160,0)(0,80){2}{\line(1,0){10}}

\put(100,-20){$B'_{n,k} - x$}
\put(90,-40){$I(B'_{n,k} - x)$}

\put(175,37){$\cdots$}
\put(200,37){$\xhookleftarrow{\mathmakebox[15pt]{}}$}

\put(200,-20){$\xhookleftarrow{\mathmakebox[15pt]{}}$}
\put(200,-40){$\xhookleftarrow{\mathmakebox[15pt]{\simeq}}$}

\multiput(230,20)(0,20){2}{\circle*{3}}
\put(230,60){\circle{3}}
\multiput(250,60)(20,0){2}{\circle*{3}}
\put(250,20){\circle{3}}
\multiput(270,0)(0,20){2}{\circle*{3}}
\put(270,80){\circle{3}}
\put(290,0){\circle{3}}
\multiput(310,0)(20,0){2}{\circle*{3}}
\multiput(290,80)(20,0){2}{\circle*{3}}
\multiput(330,20)(0,60){2}{\circle{3}}
\multiput(330,40)(0,20){2}{\circle*{3}}

\multiput(350,20)(20,0){2}{\circle*{3}}
\multiput(350,60)(20,0){2}{\circle*{3}}
\multiput(370,0)(0,80){2}{\circle*{3}}

\put(230,20){\line(0,1){20}}
\put(250,60){\line(1,0){20}}
\put(270,0){\line(0,1){20}}
\put(290,80){\line(1,0){20}}
\put(310,0){\line(1,0){20}}
\put(330,40){\line(0,1){20}}
\put(330,60){\line(1,0){40}}
\put(350,20){\line(1,0){20}}
\multiput(370,0)(0,60){2}{\line(0,1){20}}
\multiput(370,0)(0,80){2}{\line(1,0){10}}

\put(385,37){$\cdots$}

\put(250,-20){$(B'_{n,k-1} - x) \sqcup \big( \coprod_5 K_2$ \big)}
\put(270,-40){$\Sigma^5 I(B'_{n,k-1} - x)$}
\end{picture}
\caption{} \label{figure B' - x}
\end{figure}

\begin{figure}[b]
\begin{picture}(390,120)(0,-40)
\multiput(20,40)(0,20){2}{\circle{3}}
\put(40,60){\circle{3}}

\put(9,59){\footnotesize $x$}

\multiput(20,20)(20,0){3}{\circle*{3}}
\put(60,60){\circle*{3}}
\multiput(60,0)(20,0){4}{\circle*{3}}
\multiput(60,80)(20,0){4}{\circle*{3}}
\multiput(120,20)(0,20){3}{\circle*{3}}
\multiput(140,20)(20,0){2}{\circle*{3}}
\multiput(140,60)(20,0){2}{\circle*{3}}
\multiput(160,0)(0,80){2}{\circle*{3}}

\put(20,20){\line(1,0){40}}
\multiput(60,0)(0,60){2}{\line(0,1){20}}
\multiput(60,0)(0,80){2}{\line(1,0){60}}
\put(120,0){\line(0,1){80}}
\multiput(120,20)(0,40){2}{\line(1,0){40}}
\multiput(160,0)(0,60){2}{\line(0,1){20}}
\multiput(160,0)(0,80){2}{\line(1,0){10}}

\put(80,-20){$B'_{n,k} - N[x]$}
\put(70,-40){$I(B'_{n,k} - N[x])$}

\put(175,37){$\cdots$}
\put(200,37){$\xhookleftarrow{\mathmakebox[15pt]{}}$}
\put(200,-20){$\xhookleftarrow{\mathmakebox[15pt]{}}$}
\put(200,-40){$\xhookleftarrow{\mathmakebox[15pt]{\simeq}}$}

\multiput(230,20)(20,0){2}{\circle*{3}}
\multiput(230,40)(0,20){2}{\circle{3}}
\put(250,60){\circle{3}}
\multiput(270,0)(20,0){2}{\circle*{3}}
\put(270,20){\circle{3}}
\multiput(270,60)(0,20){2}{\circle*{3}}
\put(290,80){\circle{3}}
\put(310,0){\circle{3}}
\multiput(310,80)(20,0){2}{\circle*{3}}
\put(330,0){\circle{3}}
\multiput(330,20)(0,20){2}{\circle*{3}}
\put(330,60){\circle{3}}
\multiput(350,20)(20,0){2}{\circle*{3}}
\multiput(350,60)(20,0){2}{\circle*{3}}
\multiput(370,0)(0,80){2}{\circle*{3}}

\put(230,20){\line(1,0){20}}
\put(270,0){\line(1,0){20}}
\put(270,60){\line(0,1){20}}
\put(310,80){\line(1,0){20}}
\put(330,20){\line(0,1){20}}
\put(330,20){\line(1,0){40}}
\put(350,60){\line(1,0){20}}
\multiput(370,0)(0,60){2}{\line(0,1){20}}
\multiput(370,0)(0,80){2}{\line(1,0){10}}
\put(385,37){$\cdots$}

\put(250,-20){$(B'_{n,k-1} - x)\sqcup \big( \coprod_4 K_2$ \big)}
\put(270,-40){$\Sigma^4 I(B'_{n,k-1} - x)$}
\end{picture}
\caption{} \label{figure B' - N[x]}
\end{figure}

\begin{proof}[Proof of Proposition \ref{prop A - v}]
By the fold lemma, one can show that $I(A_{n,0} - v) \simeq I(X_n)$ and $I(A_{n,k} - v) \simeq I(X_n \sqcup B_{n,k}) = I(X_n) * I(B_{n,k})$ (see Figure \ref{figure A - v}). The proof is completed by Lemmas \ref{lem X} and \ref{lem B}.
\end{proof}

\begin{figure}[t]
\begin{picture}(400,80)(0,-15)
\multiput(0,0)(0,15){2}{\circle*{2}}
\multiput(0,45)(0,15){2}{\circle*{2}}
\put(0,30){\circle{2}}
\multiput(15,0)(0,15){5}{\circle*{2}}
\multiput(30,0)(0,15){5}{\circle*{2}}
\multiput(45,15)(15,0){2}{\circle*{2}}
\multiput(45,45)(15,0){2}{\circle*{2}}
\multiput(60,0)(0,60){2}{\circle*{2}}

\multiput(0,0)(0,45){2}{\line(0,1){15}}
\multiput(30,0)(0,60){2}{\line(-1,0){38}}
\multiput(60,15)(0,30){2}{\line(-1,0){68}}
\multiput(15,0)(15,0){2}{\line(0,1){60}}
\multiput(60,0)(0,45){2}{\line(0,1){15}}
\multiput(60,0)(0,60){2}{\line(1,0){8}}
\put(15,30){\line(1,0){15}}

\put(10,-18){\small $A_{n,k} - v$}

\put(76,28){\small $\xhookleftarrow{\mathmakebox[15pt]{}}$}

\multiput(110,0)(0,15){2}{\circle*{2}}
\multiput(110,45)(0,15){2}{\circle*{2}}
\multiput(125,15)(0,30){2}{\circle{2}}
\multiput(125,0)(0,30){3}{\circle*{2}}
\multiput(140,0)(0,15){5}{\circle*{2}}
\multiput(155,15)(15,0){2}{\circle*{2}}
\multiput(155,45)(15,0){2}{\circle*{2}}
\multiput(170,0)(0,60){2}{\circle*{2}}

\multiput(140,0)(0,60){2}{\line(-1,0){38}}
\multiput(110,15)(0,30){2}{\line(-1,0){8}}
\multiput(110,0)(0,45){2}{\line(0,1){15}}
\put(125,30){\line(1,0){15}}
\put(140,0){\line(0,1){60}}
\multiput(140,15)(0,30){2}{\line(1,0){30}}
\multiput(170,0)(0,45){2}{\line(0,1){15}}
\multiput(170,0)(0,60){2}{\line(1,0){8}}

\put(186,28){\small $\xhookleftarrow{\mathmakebox[15pt]{}}$}

\multiput(220,0)(0,15){2}{\circle*{2}}
\multiput(220,45)(0,15){2}{\circle*{2}}
\multiput(235,0)(0,30){3}{\circle*{2}}
\multiput(250,0)(0,30){3}{\circle*{2}}
\multiput(250,15)(0,30){2}{\circle{2}}
\multiput(265,15)(15,0){2}{\circle*{2}}
\multiput(265,45)(15,0){2}{\circle*{2}}
\multiput(280,0)(0,60){2}{\circle*{2}}

\multiput(250,0)(0,60){2}{\line(-1,0){38}}
\multiput(220,15)(0,30){2}{\line(-1,0){8}}
\multiput(220,0)(0,45){2}{\line(0,1){15}}
\put(250,30){\line(-1,0){15}}
\multiput(265,15)(0,30){2}{\line(1,0){15}}
\multiput(280,0)(0,45){2}{\line(0,1){15}}
\multiput(280,0)(0,60){2}{\line(1,0){8}}

\put(294,28){\small $\xhookrightarrow{\mathmakebox[15pt]{}}$}

\multiput(330,0)(0,15){5}{\circle*{2}}
\multiput(345,0)(0,15){5}{\circle*{2}}
\multiput(360,0)(0,30){3}{\circle*{2}}
\multiput(375,15)(0,30){2}{\circle*{2}}
\multiput(390,15)(0,30){2}{\circle*{2}}
\multiput(390,0)(0,60){2}{\circle*{2}}

\multiput(360,0)(0,30){3}{\line(-1,0){38}}
\multiput(345,15)(0,30){2}{\line(-1,0){23}}
\multiput(330,0)(15,0){2}{\line(0,1){60}}
\multiput(375,15)(0,30){2}{\line(1,0){15}}
\multiput(390,0)(0,45){2}{\line(0,1){15}}
\multiput(390,0)(0,60){2}{\line(1,0){8}}

\put(342,-18){\small $X_n \sqcup B_{n,k}$}
\end{picture}
\caption{} \label{figure A - v}
\end{figure}

\subsection{Proposition \ref{prop 1-5}}
The goal of this subsection is to prove Proposition \ref{prop 1-5}, which determines the homotopy types of $I(A_{n,k})$ for $n = 1, \cdots, 5$.

\begin{proof}[Proof of (1) of Proposition \ref{prop 1-5}]
When $k = 0$, $A_{1,k}$ is isomorphic to $\Gamma_{5,1}$ and hence $I(A_{1,0}) \simeq S^1$ follows from Theorem \ref{thm nx1}.

Suppose $k \ge 1$. The fold lemma implies
\[ I(A_{1,k}) \simeq I(A_{1,k} - \{ (2,2), (2,4), (4,1), (4,5), (6,2), (6,4)\}).\]
Since $(6,3)$ is an isolated vertex of $A_{1,k} - \{ (2,2), (2,4), (4,1), (4,5), (6,2), (6,4)\}$, the right of the above is contractible. This completes the proof of (1).
\end{proof}

\begin{proof}[Proof of (2) of Proposition \ref{prop 1-5}]
The case $k = 0$ follows from Theorem \ref{thm nx2}. For $k \ge 1$, the fold lemma implies
\[ I(A_{2,k}) \simeq I(K_2 \sqcup K_2 \sqcup K_2 \sqcup B_{n,k}) = \Sigma^3 I(B_{n,k}).\]
Thus the proof is completed by Lemma \ref{lem B}.
\end{proof}

\begin{proof}[Proof of (3) of Proposition \ref{prop 1-5}]
  Recall $v = (1,3)$ in this case. The fold lemma implies that
  \[ I(A_{3,k}- v) \simeq I(A_{3,k} - \{ (1,3), (2,2), (2,4), (3,2), (1,1)\}),\]
  and $A_{3,k} - \{ (1,3), (2,2), (2,4), (3,2), (1,1) \}$ has isolated vertex $(1,2)$. Hence $I(A_{3,k} - v)$ is contractible and hence $I(A_{3,k}) \simeq \Sigma I(A_{3,k} - N[v])$.
  % The fold lemma implies that $I(A_{3,,0} - N[v]) \simeq S^2$, $I(A_{3,1} - N[v]) \simeq I(C_{18})$,
  % and $I(A_{3,k} - N[v]) \simeq \Sigma^2 I(B'_{n,k-1})$ for $k \ge 2$.
  The fold lemma implies that $I(A_{3,k} - N[v]) \simeq \Sigma^2 I(B'_{n,k+1})$ for $k \ge 0$.
  Hence the proof is completed by Lemma \ref{lem B}.
\end{proof}

\begin{proof}[Proof of (4) of Proposition \ref{prop 1-5}]
The fold lemma implies that
\[ I(A_{n,k} - N[v]) \simeq I(A_{4,k} - (N[v] \cup \{ (2,1), (2,5), (4,2), (4,4)\})).\]
Since $(4,3)$ is an isolated vertex of $A_{4,k} - (N[v] \cup \{ (2,1), (2,5), (4,2), (4,4)\}$, this means that $I(A_{4,k} - N[v])$ is contractible. Thus Theorem \ref{thm cofiber} implies that $I(A_{4,k}) \simeq I(A_{4,k} - v)$. It follows from the fold lemma that $I(A_{4,0} - v) \simeq S^4$ and $I(A_{4,k}) \simeq \Sigma^5 I(B_{4,k})$ for $k \ge 1$. The proof is completed by Lemma \ref{lem B}.
\end{proof}

\begin{figure}[b]
\begin{picture}(260,120)(0,-20)
\multiput(60,0)(0,20){5}{\circle*{3}}
\multiput(-20,0)(0,20){5}{\circle*{3}}
\multiput(0,0)(20,0){3}{\circle*{3}}
\multiput(0,80)(20,0){3}{\circle*{3}}

\multiput(80,20)(20,0){2}{\circle*{3}}
\multiput(80,60)(20,0){2}{\circle*{3}}
\multiput(100,0)(20,0){4}{\circle*{3}}
\multiput(100,80)(20,0){4}{\circle*{3}}

\multiput(160,20)(0,20){3}{\circle*{3}}
\multiput(180,20)(20,0){2}{\circle*{3}}
\multiput(180,60)(20,0){2}{\circle*{3}}
\multiput(200,0)(20,0){4}{\circle*{3}}
\multiput(200,80)(20,0){4}{\circle*{3}}
\multiput(260,20)(0,20){3}{\circle*{3}}

\put(-20,0){\line(0,1){80}}
\multiput(-20,0)(0,80){2}{\line(1,0){80}}

\put(60,0){\line(0,1){80}}
\multiput(60,20)(0,40){2}{\line(1,0){40}}
\multiput(100,0)(0,60){2}{\line(0,1){20}}
\multiput(100,0)(0,80){2}{\line(1,0){60}}
\multiput(160,0)(100,0){2}{\line(0,1){80}}
\multiput(200,0)(0,80){2}{\line(1,0){60}}
\multiput(160,20)(0,40){2}{\line(1,0){40}}
\multiput(200,0)(0,60){2}{\line(0,1){20}}
\multiput(260,20)(0,40){2}{\line(1,0){10}}
\put(280,37){$\cdots$}

\put(120,-23){$W_k$}
\end{picture}
\caption{} \label{figure W}
\end{figure}

\begin{proof}[Proof of (5) of Proposition \ref{prop 1-5}]
Since $I(A_{5,k} - v) \simeq {\rm pt}$, we have $I(A_{n,k}) \simeq \Sigma I(A_{5,k}- N[v])$. The fold lemma implies that $I(A_{5,k} - N[v])$ is isomorphic to the graph $W_k$ depicted in Figure \ref{figure W}, and hence
\begin{eqnarray} \label{W}
I(A_{5,k}) & \simeq & \Sigma I(W_k)
\end{eqnarray}
 % Note that $W_0 = C_{16}$ and hence Theorem \ref{thm cycle} implies $I(W_0) \simeq S^4$, and $I(A_{5,0}) \simeq S^5$.\todo{$B'_{n,1}$を使えばここの場合分けは不要}

Set $u = (1,3)$. Then
\[ I(W_k - N[v]) \simeq I(W_k - (N[v] \cup \{ (3,1), (3,5), (5,2), (5,3)\})).\]
Since $(5,3)$ is an isolated vertex in $W_k - (N[v] \cup \{ (3,1), (3,5), (5,2), (5,3)\})$, we have $I(W_k - N[v]) \simeq {\rm pt}$. This shows
\[ I(W_k) \simeq I(W_k - u).\]
The fold lemma shows that $I(W_k - u) \simeq \Sigma^4 I(B'_{5,k+1})$. Lemma \ref{lem B} completes the proof.
\end{proof}

\subsection{Proposition \ref{prop A - N[v]}}

\begin{figure}[t]
\begin{picture}(330,120)(0,-40)
\multiput(40,0)(20,0){6}{\circle*{3}}
\multiput(40,20)(20,0){3}{\circle*{3}}
\multiput(40,40)(20,0){2}{\circle*{3}}
\multiput(80,40)(20,0){3}{\circle{3}}
\multiput(40,60)(20,0){3}{\circle*{3}}
\multiput(40,80)(20,0){6}{\circle*{3}}
\multiput(100,20)(0,40){2}{\circle{3}}
\multiput(140,20)(0,20){3}{\circle*{3}}
\multiput(140,20)(0,40){2}{\line(1,0){10}}

\multiput(40,0)(20,0){2}{\line(0,1){80}}
\multiput(80,0)(0,60){2}{\line(0,1){20}}
\put(140,0){\line(0,1){80}}
\multiput(80,20)(0,40){2}{\line(-1,0){50}}
\put(60,40){\line(-1,0){30}}

\multiput(140,0)(0,80){2}{\line(-1,0){110}}

\multiput(120,20)(0,40){2}{\circle*{3}}
\multiput(120,20)(0,40){2}{\line(1,0){20}}
\multiput(120,0)(0,60){2}{\line(0,1){20}}

\put(55,-20){$A_{n,k} - N[v]$}
\put(45,-40){$I(A_{n,k} - N[v])$}

\put(10,37){$\cdots$}

\multiput(210,0)(0,20){5}{\circle*{3}}
\multiput(230,0)(0,80){2}{\circle{3}}
\multiput(230,20)(0,20){3}{\circle*{3}}
\multiput(250,0)(20,0){4}{\circle*{3}}
\multiput(250,80)(20,0){4}{\circle*{3}}
\multiput(250,20)(0,40){2}{\circle*{3}}
\multiput(310,20)(0,20){3}{\circle*{3}}

\multiput(310,20)(0,40){2}{\line(1,0){10}}

\multiput(290,20)(0,40){2}{\circle{3}}

\multiput(210,0)(100,0){2}{\line(0,1){80}}
\put(230,20){\line(0,1){40}}
\multiput(250,0)(0,60){2}{\line(0,1){20}}
\multiput(210,20)(0,40){2}{\line(1,0){40}}
\multiput(250,0)(0,80){2}{\line(1,0){60}}
\put(230,40){\line(-1,0){20}}
\multiput(210,0)(0,20){5}{\line(-1,0){10}}

\put(240,-20){$A'$}
\put(230,-40){$I(A')$}
\put(180,37){$\cdots$}

\put(155,37){$\xhookleftarrow{\mathmakebox[15pt]{}}$}
\put(155,-20){$\xhookleftarrow{\mathmakebox[15pt]{}}$}
\put(155,-40){$\xhookleftarrow{\mathmakebox[15pt]{\simeq}}$}
\put(235,38){\footnotesize $w$}

\end{picture}
\caption{} \label{figure A'}
\end{figure}

The goal of this subsection is to prove Proposition \ref{prop A - N[v]}, which states that $I(A_{n,k} - N[v]) \simeq I(A_{n-5, k+1})$ for $n \ge 6$.

\begin{proof}[Proof of Proposition \ref{prop A - N[v]}]
Set $w = (n-4, 3)$. The fold lemma implies that $I(A_{n,k} - N[v])$ is homotopy equivalent to the independence complex of the induced subgraph $A'$ of $A_{n,k}$ whose vertex set is $V(A_{n-5,k+1}) \cup \{ w \}$ (see Figure \ref{figure A'}). Since $A' - w = A_{n-5, k+1}$, it suffices to show that $I(A' - N[w])$ is contractible (Theorem \ref{thm cofiber}).

The fold lemma implies
\[ I(A' - N[w]) \simeq I(A' - (N[w] \cup \{ (n-2,5), (n,4), (n,1), (n-3,1) \})).\]
Since $(n-3,2)$ is an isolated vertex of $A' - (N[w] \cup \{ (n-2,5), (n,4), (n,1), (n-3,1) \}$, we conclude that $I(A' - N[w])$ is contractible (see Figure \ref{figure A' - N[w]}).
\end{proof}

\subsection{Proposition \ref{prop null-homotopic}}

In this subsection, we prove Proposition \ref{prop null-homotopic}, which states that the map $I(A_{n-5,k+1}) \simeq I(A_{n,k} - N[v]) \hookrightarrow I(A_{n,k} - v)$ is null-homotopic for $n \ge 6$. Note that $I(A_{n,k} - v) \simeq {\rm pt}$ when $n$ is odd, and that in this case $I(A_{n,k} - N[v]) \to I(A_{n,k} - v)$ is null-homotopic. Hence we prove only in the case that $n$ is even.

\begin{figure}[t]
\begin{picture}(330,120)(0,-20)
\put(-10,37){$\cdots$}

\multiput(20,0)(0,20){2}{\circle*{3}}
\put(20,40){\circle{3}}
\multiput(20,60)(0,20){2}{\circle*{3}}
\multiput(40,20)(0,20){3}{\circle{3}}
\multiput(60,0)(0,20){2}{\circle*{3}}
\multiput(60,60)(0,20){2}{\circle*{3}}
\multiput(80,0)(20,0){3}{\circle*{3}}
\multiput(80,80)(20,0){3}{\circle*{3}}
\multiput(120,20)(0,20){3}{\circle*{3}}

\multiput(20,0)(0,20){2}{\line(-1,0){10}}
\multiput(20,0)(0,60){2}{\line(0,1){20}}
\multiput(20,60)(0,20){2}{\line(-1,0){10}}
\multiput(60,0)(0,60){2}{\line(0,1){20}}
\multiput(60,0)(0,80){2}{\line(1,0){60}}
\put(120,0){\line(0,1){80}}
\multiput(120,20)(0,40){2}{\line(1,0){10}}

\put(135,37){$\cdots$}

\put(40,-20){$A' - N[w]$}

\put(155,37){$\xhookleftarrow{\mathmakebox[15pt]{}}$}

\put(180,37){$\cdots$}

\multiput(210,0)(0,20){2}{\circle*{3}}
\multiput(210,60)(0,20){2}{\circle*{3}}
\multiput(250,60)(0,20){2}{\circle*{3}}
\put(270,80){\circle{3}}
\multiput(290,80)(20,0){2}{\circle*{3}}
\put(310,60){\circle{3}}
\multiput(310,20)(0,20){2}{\circle*{3}}
\put(310,0){\circle{3}}
\multiput(270,0)(20,0){2}{\circle*{3}}
\put(250,0){\circle{3}}
\put(250,20){\circle*{3}}

\multiput(210,0)(0,20){2}{\line(-1,0){10}}
\multiput(210,60)(0,20){2}{\line(-1,0){10}}
\multiput(210,0)(0,60){2}{\line(0,1){20}}
\put(250,60){\line(0,1){20}}
\put(290,80){\line(1,0){20}}
\put(310,20){\line(0,1){20}\line(1,0){10}}
\put(270,0){\line(1,0){20}}

\put(330,37){$\cdots$}
\end{picture}
\caption{} \label{figure A' - N[w]}
\end{figure}

\begin{prop} \label{prop key}
Assume that $n$ is even. Then the following hold:
\begin{enumerate}[$(1)$]
\item If $n = 4l$, $I(A_{4l,k})$ is homotopy equivalent to a wedge of spheres whose dimension is at most $5k+5l-1$. If $n = 4l+2$, then $I(A_{4l+2,k})$ is homotopy equivalent to a wedge of spheres whose dimension is at most $5k+5l+2$.

\item The inclusion $I(A_{n,k} - N[v]) \to I(A_{n,k} - v)$ is null-homotopic.
\end{enumerate}
\end{prop}
\begin{proof}
We simultaneously prove (1) and (2) by induction on $n$. Since the cases $n = 6,8,10$ can be directly proved by Propositions \ref{prop A - N[v]}, \ref{prop A - v}, and \ref{prop 1-5}, we assume that $n \ge 12$. Then we have a cofiber sequence
\[ I(A_{n-5,k+1}) \to I(A_{n,k} - v) \to I(A_{n,k})\]
and
\[ I(A_{n-10,k+2}) \to I(A_{n-5,k+1} - v) \to I(A_{n-5,k+1}).\]
Since $n-5$ is odd, we have $I(A_{n-5,k+1} - v) \simeq {\rm pt}$, and hence we have $I(A_{n-5, k+1}) \simeq \Sigma I(A_{n-10,k+2})$. To prove (2), it suffices to see that every map from $\Sigma I(A_{n-10, k+2})$ to $I(A_{n,k} - v)$ is null-homotopic.

Suppose that $n = 4l$. Then Proposition \ref{prop A - v} implies that $I(A_{n,k} - v) = I(A_{4l,k} - v)$ is a wedge of spheres whose dimension is $5k+5l-1$. It follows from the inductive hypothesis that $\Sigma I(A_{n-10,k+2}) = \Sigma I(A_{4(l-3)+2, k+2})$ is a wedge of spheres whose dimension is at most
\[ 5(l-3) + 5(k+2) + 2 + 1 = 5k + 5l - 2.\]
Hence every map from $\Sigma I(A_{4l-10,k+2})$ to $I(A_{4l,k} - v)$ is null-homotopic, and we have proved (2) in the case $n = 4l$. Since $I(A_{4l,k}) \simeq I(A_{4l,k} - v) \vee \Sigma^2 I(A_{4l - 10, k+2})$, we conclude that $I(A_{4l,k})$ is a wedge of spheres whose dimension is at most $5k+5l-1$. This completes the proof of (1) in the case $n = 4l$.

Next suppose that $n = 4l + 2$. Then Proposition \ref{prop A - v} implies that $I(A_{n,k} - v) = I(A_{4l+2,k})$ is a wedge of spheres whose dimension is $5k+5l+2$. It follows from the inductive hypothesis that $\Sigma I(A_{n -10, k+2}) = \Sigma I(A_{4(l-2), k+2})$ is a wedge of spheres whose dimension is at most
\[ 5(l-2) + 5(k+2) - 1 + 1 = 5k + 5l.\]
Hence we have proved (2) in the case $n = 4l+2$. Since
\[ I(A_{4l+2,k}) \simeq I(A_{4l+2,k} - v) \vee \Sigma^2 I(A_{4l-8, k+2}),\]
we conclude that $I(A_{4l+2,k})$ is a wedge of spheres whose dimension is at most $5k+5l+2$. This completes the proof of (1) in the case $n = 4l+2$, and we have completed the proof of this proposition.
\end{proof}

\begin{proof}[Proof of Proposition \ref{prop null-homotopic}]
The proof is completed by Propositions \ref{prop A - N[v]} and \ref{prop key}.
\end{proof}

\subsection{Proof of the main theorem}

In this subsection we complete the proof of Theorem \ref{main thm}, using the results in the previous subsections.

\begin{lem} \label{lem 1}
The following hold:
\begin{enumerate}[(1)]
\item If $n$ is odd and $n \ge 7$, then $I(A_{n,k}) \simeq \Sigma I(A_{n-5,k+1})$.

\item If $n$ is even and $n \ge 12$, then $I(A_{n,k}) \simeq I(A_{n,k} - v) \vee \Sigma^2 I(A_{n-10,k+2})$.
\end{enumerate}
\end{lem}
\begin{proof}
Suppose $n$ is odd. It follows from Proposition \ref{prop A - v} that $I(A_{n,k} - v)$ is contractible. Hence Theorem \ref{thm cofiber} and Proposition \ref{prop A - N[v]} imply $I(A_{n,k}) \simeq \Sigma I(A_{n,k} - N[v]) \simeq \Sigma I(A_{n-5, k+1})$. This completes the proof of (1).

Suppose that $n$ is even and $n \ge 12$. By Proposition \ref{prop null-homotopic} we have
\begin{eqnarray*}
I(A_{n,k}) & \simeq & I(A_{n,k} - v) \vee \Sigma I(A_{n-5, k+1}) \\
& \simeq & I(A_{n,k} - v) \vee \Sigma^2 I(A_{n-10,k+2})
\end{eqnarray*}
Here the second homotopy equivalence follows from (1). This completes the proof.
\end{proof}

\begin{cor} \label{cor 2}
There is a following homotopy equivalence
\[ I(A_{4l,k}) \simeq \begin{cases}
\displaystyle\Big( \bigvee_3 S^{5k+5l-1}\Big) \vee \Sigma^4 I(A_{4(l-5), k+4}) & (k=0,1)\\
\displaystyle \Big( \bigvee_4 S^{5k+5l-1}\Big) \vee \Sigma^4 I(A_{4(l-5), k+4}) & (k \ge 2).
\end{cases}\]
\end{cor}
\begin{proof}
Lemma \ref{lem 1} implies
\begin{eqnarray*}
I(A_{4l, k}) & \simeq & I(A_{4l,k} - v) \vee \Sigma^2 I(A_{4l - 10, k+2}) \\
&\simeq & I(A_{4l,k} - v) \vee \Sigma^2 I(A_{4l - 10, k+2} - v) \vee \Sigma^4 I(A_{4l -20, k+4}).
\end{eqnarray*}
Proposition \ref{prop A - v} completes the proof.
\end{proof}

\begin{lem} \label{lem 3}
The following hold:
\begin{enumerate}[$(1)$]
\item
\[ I(A_{4,k}) \simeq \begin{cases}
S^{5k+4} & (k= 0,1) \\
\bigvee_2 S^{5k+4} & (k \ge 2).
\end{cases}\]
\item For $l = 2,3,4$, there is a following homotopy equivalence:
\[ I(A_{4l,k}) \simeq \begin{cases}
\bigvee_3 S^{5k+5l - 1} & (k = 0,1) \\
\bigvee_4 S^{5k+5l - 1} & (k \ge 2).
\end{cases}\]

\item
\[ I(A_{20,k}) \simeq \begin{cases}
\bigvee_3 S^{5k+24} \vee \bigvee_2 S^{5k+23} & (k = 0,1) \\
\bigvee_4 S^{5k+24} \vee \bigvee_2 S^{5k+23} & (k \ge 2).
\end{cases}\]
\end{enumerate}
\end{lem}
\begin{proof}
We have already proved (1) in Proposition \ref{prop 1-5}. To prove (2) and (3), we use the following homotopy equivalences
\[ I(A_{8,k}) \simeq I(A_{8,k} - v) \vee \Sigma I(A_{3,k+1}),\]
\[ I(A_{12,k}) \simeq I(A_{12,k} - v) \vee \Sigma^2 I(A_{2,k+2}),\]
\[ I(A_{16,k}) \simeq I(A_{16,k} - v) \vee \Sigma^2 I(A_{6,k+2}) \simeq I(A_{16,k} - v) \vee \Sigma^2 I(A_{6,k+2} - v) \vee \Sigma^3 I(A_{1,k+3}),\]
\[ I(A_{20,k}) \simeq I(A_{20,k} - v) \vee \Sigma^2 I(A_{10,k+2}) \simeq I(A_{20,k} - v) \vee \Sigma^2 I(A_{10,k+2} - v) \vee \Sigma^3 I(A_{5,k+3}).\]
Here we use Proposition \ref{prop null-homotopic} and Lemma \ref{lem 1} to derive these homotopy equivalences. Then Propositions \ref{prop A - v} and \ref{prop 1-5} complete the proof.
\end{proof}

When $n$ is even,
the homotopy type of $I(A_{n, k})$ is determined by the following two propositions.
Here we write $n = 20t + 4l + \varepsilon$ with $l\in\{0,1,2,3,4\}$ and $\varepsilon\in\{0,1,2,3\}$.
Define $\kappa\in\{0,1\}$ by $0$ when $k = 0, 1$ and by $1$ when $k \ge 2$.

%\begin{prop} \label{prop 4}
%  Consider the case $\varepsilon = 0$ and assume $n \ge 8$.
%  Let $n' = 25t + 5l + 5k - 1$.
%  \begin{enumerate}[$(1)$]
%    \item For $l = 0, 1$ (and hence $t \ge 1$), we have
%      \[
%      I(A_{20t+4l, k}) \simeq \bigvee_{3+\kappa}S^{n'}
%     \vee \Big( \bigvee_{n'-t< i<n'} \Big( \bigvee_4S^i \Big) \Big)
%     \vee \bigvee_2S^{n'-t}.
%      \]
%    \item For $l = 2, 3, 4$, we have
%      \[
%      I(A_{20t+4l, k}) \simeq \bigvee_{3+\kappa}S^{n'}
%      \vee \Big( \bigvee_{n'-t\le i<n'} \Big( \bigvee_4S^i \Big) \Big).
%      \]
%  \end{enumerate}
%\end{prop}

\begin{prop} \label{prop 4}
  Consider the case $\varepsilon = 0$ and assume $n \ge 8$. Then the following hold:
  \begin{enumerate}[$(1)$]
    \item For $l = 0, 1$ (and hence $t \ge 1$), we have
      \[
      I(A_{20t+4l, k}) \simeq \bigvee_{3+\kappa}S^{25t + 5l + 5k - 1}
      \vee \Big( \bigvee_{i = 24t + 5l + 5k}^{25 t + 5l + 5k - 2} \Big( \bigvee_4 S^i \Big) \Big)
      \vee \bigvee_2 S^{24t + 5l + 5k - 1}.
      \]
    \item For $l = 2, 3, 4$, we have
      \[
      I(A_{20t+4l, k}) \simeq \bigvee_{3+\kappa}S^{25t + 5l + 5k - 1}
      \vee \Big( \bigvee_{i = 24t + 5l + 5k - 1}^{25t + 5l + 5k - 2} \Big( \bigvee_4S^i \Big) \Big).
      \]
  \end{enumerate}
\end{prop}
\begin{proof}
Corollary \ref{cor 2} implies
\[ I(A_{4l + 20t,k}) \simeq \begin{cases}
\Big( \bigvee_3 S^{25t + 5l + 5k - 1} \Big) \vee \Sigma^4 I(A_{20(t-1)+4l, k+4}) & (k=0,1)\\
\Big( \bigvee_4 S^{25t + 5l + 5k - 1} \Big) \vee \Sigma^4 I(A_{20(t-1)+4l, k+4}) & (k \ge 2).
\end{cases}\]
Lemma \ref{lem 3} completes the proof.
\end{proof}

%\begin{prop} \label{prop 5}
%  Consider the case $\varepsilon = 2$.
%  Let $n' = 25t + 5l + 5k + 2$.
%  For any $n\ge 2$ (i.e.\ $t,l\ge 0$), the following hold:
%  \begin{enumerate}[$(1)$]
%    \item For $l = 0, 1$, we have
%      \[
%      I(A_{20t+4l+2, k}) \simeq \bigvee_{1+\kappa}S^{n'}
%      \vee \Big( \bigvee_{n'-t\le i<n'} \Big( \bigvee_4S^i \Big) \Big).
%      \]
%    \item For $l = 2, 3$, we have
%      \[
%      I(A_{20t+4l+2, k}) \simeq \bigvee_{1+\kappa}S^{n'}
%      \vee \Big( \bigvee_{n'-t\le i<n'} \Big( \bigvee_4 S^i \Big) \Big)
%      \vee \bigvee_2S^{n'-t-1}.
%      \]
%    \item For $l = 4$, we have
%      \[
%      I(A_{20t+4l+2, k}) \simeq \bigvee_{1+\kappa}S^{n'}
%      \vee \Big( \bigvee_{n'-t-1\le i<n'} \Big( \bigvee_4S^i \Big) \Big).
%      \]
%  \end{enumerate}
%\end{prop}

\begin{prop} \label{prop 5}
  Consider the case $\varepsilon = 2$. For every $n\ge 2$ (i.e.\ $t,l\ge 0$), the following hold:
  \begin{enumerate}[$(1)$]
    \item For $l = 0, 1$, we have
      \[
      I(A_{20t+4l+2, k}) \simeq \bigvee_{1+\kappa}S^{25t + 5l + 5k + 2}
      \vee \Big( \bigvee_{i = 24t + 5l + 5k + 2}^{25t + 5l + 5k + 1} \Big( \bigvee_4S^i \Big) \Big).
      \]
    \item For $l = 2, 3$, we have
      \[
      I(A_{20t+4l+2, k}) \simeq \bigvee_{1+\kappa}S^{25t + 5l + 5k + 2}
      \vee \Big( \bigvee_{i = 24t + 5l + 5k + 2}^{25t + 5l + 5k + 1} \Big( \bigvee_4 S^i \Big) \Big)
      \vee \bigvee_2 S^{24t + 5l + 5k + 1}.
      \]
    \item For $l = 4$, we have
      \[
      I(A_{20t+4l+2, k}) \simeq \bigvee_{1+\kappa}S^{25t + 5l + 5k + 2}
      \vee \Big( \bigvee_{i = 24 t + 5l + 5k + 1}^{25t + 5l + 5k + 1} \Big( \bigvee_4S^i \Big) \Big).
      \]
  \end{enumerate}
\end{prop}
\begin{proof}
  We have already determined the homotopy type of $I(A_{2,k})$ in Proposition \ref{prop 1-5}.
  Proposition \ref{prop null-homotopic} and Lemma \ref{lem 1} imply
  \[ I(A_{6,k}) \simeq I(A_{6,k} - v) \vee \Sigma I(A_{1,k+1}),\]
  \[ I(A_{10,k}) \simeq I(A_{10,k} - v) \vee \Sigma I(A_{5,k+1}),\]
  \[ I(A_{14,k}) \simeq I(A_{14,k} - v) \vee \Sigma^2 I(A_{4,k+2}). \]
  Thus the homotopy types of $I(A_{6,k})$, $I(A_{10,k})$ and $I(A_{14, k})$ are determined by Propositions \ref{prop A - v} and \ref{prop 1-5}.
  % Other homotopy equivalences are deduced from (2) of Lemma \ref{lem 1} and Proposition \ref{prop 4}.
  For $n \ge 18$, Lemma \ref{lem 1} implies
  $I(A_{n, k})\simeq I(A_{n,k}-v)\vee\Sigma^2I(A_{n-10,k+2})$.
  Its homotopy type is determined in Proposition \ref{prop 4} since $n-10 \ge 8$.
\end{proof}

\begin{proof}[Proof of Theorem \ref{main thm}]
Since $\Gamma_n = A_{n,0}$, the homotopy type of $I(\Gamma_n)$ for even $n$ is determined by Propositions \ref{prop 4} and \ref{prop 5}. The homotopy types of $I(\Gamma_1)$, $I(\Gamma_3)$, and $I(\Gamma_5)$ are determined in Proposition \ref{prop 1-5}. For an odd number $n$ with $n > 5$, (1) of Lemma \ref{lem 1} implies
\[ I(\Gamma_n) \simeq \Sigma I(A_{n-5,1}),\]
and the homotopy type of $I(A_{n-5,1})$ is determined by Propositions \ref{prop 4} and \ref{prop 5}. This completes the proof.
\end{proof}

%\begin{prop}
%There are following homotopy equivalences:
%\[ I(\Gamma_1) \simeq S^1, \; I(\Gamma_2) \simeq S^2, \; I(\Gamma_3) \simeq S^3, \; I(\Gamma_4) \simeq S^4, \; I(\Gamma_5) \simeq S^5,\; I(\Gamma_6) \simeq S^7, I(\Gamma_7) \simeq S^8,\]
%\[ I(\Gamma_9) \simeq S^{10}, I(\Gamma_{10}) \simeq S^{12} \vee \bigvee_2 S^{11},\; I(\Gamma_{11}) \simeq S^{13}, \; I(\Gamma_{12}) \simeq \bigvee_3 S^{14},\]
%\[ I(\Gamma_{13}) \simeq \bigvee_3 S^{15}, \; I(\Gamma_{14}) \simeq S^{17} \vee \bigvee_2 S^{16}, \; I(\Gamma_{15}) \simeq \bigvee_3 S^{18},\; I(\Gamma_{16}) \simeq \bigvee_3 S^{19},\]
%\[ I(\Gamma_{17}) \simeq \bigvee_3 S^{20}, \; I(\Gamma_{18}) \simeq S^{22} \vee \bigvee_4 S^{21}, \; I(\Gamma_{19}) \simeq S^{23} \vee \bigvee_2 S^{22}, \; I(\Gamma_{20}) \simeq \bigvee_3 S^{24} \vee \bigvee_2 S^{23}\]
%\end{prop}

\section*{Acknowledgment}
The authors thank the referee for useful comments. The first author was partially supported by JSPS KAKENHI Grant Number JP19K14536 and 23K12975. The second author was supported by JSPS KAKENHI Grant Number JP20J00404.

\end{document}